\documentclass[12pt]{amsart}
 \usepackage[margin=1.2in]{geometry} 
\usepackage{amsmath,amsthm,amssymb,amsfonts, color, comment, graphicx, environ, mathrsfs,hyperref,blindtext,tikz}

\usepackage[utf8]{inputenc}
\usepackage[english]{babel}

\newtheorem{theorem}{Theorem}

\newtheorem{corollary}[theorem]{Corollary}
\newtheorem{lemma}[theorem]{Lemma}

\newtheorem{fact}[theorem]{Fact}
\newtheorem{remark}[theorem]{Remark}

\theoremstyle{definition}
\newtheorem{definition}[theorem]{Definition}

\newcommand{\E}{\mathbb{E}}
\newcommand{\prob}{\mathbb{P}}

\newcommand{\GSz}{G_{Sz}}
\newcommand{\whp}{{\rm whp}}
\newcommand{\chicr}{\chi_{\rm cr}}

\title{Tiling randomly perturbed bipartite graphs}
\author{Enrique Gomez-Leos}
\thanks{Research partially supported by National Science Foundation RTG grant DMS-1839918. \texttt{enriqueg@iastate.edu}} 
\author{Ryan R. Martin}
\thanks{Research partially supported by Simons Foundation Collaboration Grant for Mathematicians \#709641. \texttt{rymartin@iastate.edu}} 
\address{Iowa State University, Ames, Iowa, USA} 

\begin{document}

\maketitle

\begin{abstract}
    A perfect $H$-tiling in a graph $G$ is a collection of vertex-disjoint copies of a graph $H$ in $G$ that covers all vertices of $G$. Motivated by papers of Bush and Zhao and of Balogh, Treglown, and Wagner,  we determine the threshold for the existence of a perfect $K_{h,h}$-tiling of a randomly perturbed bipartite graph with linear minimum degree. 
\end{abstract}

\section{Introduction}
Embedding problems form a central part of both extremal and random graph theory. Many results in extremal graph theory concern minimum degree conditions that guarantee the existence of some spanning subgraph. A celebrated result of Corr\'{a}di and Hajnal gives the minimum degree of finding a perfect $K_3$-tiling \cite{corradi1963maximal}. Later, Hajnal and Szemer\'{e}di generalized this result to cliques of arbitrary size~\cite{hajnal1970proof}. Since then, there have been generalizations to the multipartite setting, for instance  \cite{zhao2009bipartite, magyar2002tripartite, martin2008quadripartite, keevash2015multipartite}. 

Recall that the Erd\H{o}s-R\'{e}nyi random graph $G_{n,p}$ consists of the vertex set $[n]$ where each edge is present, independently, with probability $p$. In this regime, a key question is to establish the threshold for which $G_{n,p}$ contains a spanning subgraph.  
An \emph{$H$-tiling} of a graph $G$ is a subgraph consisting of vertex disjoint copies of $H$ and a \emph{perfect $H$-tiling} of $G$ is an $H$-tiling which spans all vertices of $G$. 
In their groundbreaking paper, Johansson, Kahn, and Vu~\cite{JKV} settled the threshold for which $G_{n,p}$ admits a perfect $H$-tiling for a fixed nonempty, \textit{strictly balanced}, graph $H$. 
More recently, Gerke and McDowell \cite{GerkeMcDowell} determined the corresponding threshold when $H$ is what they call \textit{nonvertex-balanced} graph.

Bohman, Frieze, and Martin~\cite{bohman2003many} introduced the randomly perturbed graph model which connects the two questions together. In the randomly perturbed setting, Balogh, Treglown, and Wagner \cite{balogh2019tilings} determined a probability $p$ for the appearance of a perfect $H$-tiling in a graph on $n$ vertices with minimum degree at least $\alpha n$ and showed that this is best possible for $\alpha<1/|V(H)|$~\cite[Section 2.1]{balogh2019tilings}, where $H$ is strictly balanced.  


In this paper, we let $h\geq 1$ be a positive integer, set $H=K_{h,h}$, and consider tiling the \emph{randomly perturbed bipartite graph} which consists of two graphs on the same vertex set $A\sqcup B$, $|A|=|B|=n$, and $h \mid n$. One is an arbitrary graph $G_{n;\alpha}$, the other is a random graph $G_{n,n,p}$, each of which is bipartite, and each respects the same bipartition $(A,B)$. Hence each of $A$ and $B$ will always be independent sets. A bipartite graph in which each class is the same size is called \textit{balanced}.

In this setting, $G_{n;\alpha}$ is an arbitrary  bipartite graph as above with minimum degree $\delta(G) \geq \alpha n$ and $G_{n,n,p}$ is an Erd\H{o}s-R\'{e}nyi random bipartite graph as above, respecting the same bipartition as $G_{n;\alpha}$. Note that that we differ from Balogh, Treglown, and Wagner~\cite{balogh2019tilings} in that we restrict the deterministic graph, $G_{n;\alpha}$, to be bipartite but we also restrict the appearance of the random edges to appear only between the two vertex disjoint classes. 

Zhao~\cite{zhao2009bipartite} proved that, for $h\geq 2$, if a balanced bipartite graph on $2n$ vertices, $n$ divisible by $h$, has minimum degree 
\[
    \delta(G) \geq \begin{cases}
    \frac{n}{2} +h -1, & n/h\text{ even}\\
    \frac{n+3h}{2}-2, & n/h\text{ odd},
    \end{cases}
\]
and $n$ sufficiently large, then it has a perfect $K_{h,h}$-tiling with no random edges necessary. Later, the case of tiling by unbalanced complete bipartite graphs was settled by Hladk\'y and Schacht~\cite{Hladkytiling} and by Czygrinow and DeBiasio~\cite{Czygrinowtiling}. Bush and Zhao~\cite{Bush} later covered the remaining cases.

Note that for any bipartite graph $H$ on $h$ vertices, it suffices to tile a graph with copies of $K_{h,h}$ in order to obtain an $H$-tiling. This is because $K_{h,h}$ itself can be tiled with two copies of $H$. Although weaker conditions could suffice for a perfect $H$-tiling, where $H$ is any bipartite graph, we direct only our attention to the the case of when $H=K_{h,h}$. 
\begin{remark}\label{remark:divisibility}
  When $h$ does not divide $n$, then it is not possible to tile all of the vertices of $G_{n;\alpha}$ by copies of $K_{h,h}$. Since our result is asymptotic, if $n$ is not divisible by $h$, then we can delete at most $h-1$ vertices from each part and prove the theorem for the resulting graph. Since the value of $n$ differs by at most a constant, the threshold function will not change. In sum, we will always assume that $n$ is divisible by $h$.   
\end{remark}

We use $\prob(A)$ to denote the probability of event $A$ and use $\E[X]$ to denote the expectation of random variable $X$. We say that a sequence of events $A_1,A_2,\ldots,A_n,\ldots$ occurs \emph{with high probability} ($\whp$) if $\lim_{n\to \infty} \prob (A_n) = 1$.
In the perturbed graph setting, we define the threshold as follows: 

The function $t(n): \mathbb{N} \to (0,1)$ is a \emph{threshold function for property $P$} if there exist real positive constants $c, C$ such that 
\begin{align*}
    p(n) \leq ct(n) &\qquad\Longrightarrow\qquad \exists G_{n;\alpha}^{*} \mbox{ such that } G_{n;\alpha}^{*}\cup G_{n,n,p}\not\in P, \whp. \\
    p(n) \geq  C t(n)&\qquad\Longrightarrow\qquad  \forall G_{n;\alpha} \mbox{ we have } G_{n;\alpha }\cup G_{n,n,p}\in P, \whp.
    \end{align*}


The question of interest is to determine the threshold function for a perfect $K_{h,h}$-tiling in $G_{n;\alpha} \cup G_{n,n,p}$, given that $h \mid n$. Bollob\'as and Thomason~\cite{BTthresholds} established that such a threshold exists. Our main result is as follows. 

\begin{theorem}\label{thm:threshold}
    For each real number $\alpha \in \bigl(0,1/(2h)\bigr)$, and positive integer $h \geq 1$, a threshold function for the existence of a perfect $K_{h,h}$-tiling is $n^{-(2h-1)/h^2}$.
\end{theorem}

The statement of Theorem~\ref{thm:threshold} is proved in two parts. These are Theorem~\ref{thm: example} and Theorem~\ref{thm: main} below.
\begin{theorem}\label{thm: example}
    For each real number $\alpha \in \bigl(0,1/(2h)\bigr)$ and positive integer $h\geq 1$, there exists a constant $c>0$ and a sequence of graphs $\bigl\{G_{n;\alpha}^{*}\bigr\}_{n\in h\mathbb{N}}$ for which $G_{n;\alpha}^{*}$ is a balanced bipartite graph on $2n$ vertices, with minimum degree $\delta\bigl(G_{n;\alpha}^{*}\bigr) \geq \alpha n$. The sequence $\bigl\{G_{n;\alpha}^{*}\bigr\}$ satisfies the condition that if $p\leq cn^{-(2h-1)/h^2}$ then, \whp, $G_{n;\alpha}^{*}\cup G_{n,n,p}$ does not contain a perfect $K_{h,h}$-tiling.
\end{theorem}

\begin{theorem}\label{thm: main}
    For each real number $\alpha >0$ and positive integer $h\geq 1$, there exists a constant $C>0$ such that any sequence of graphs $\bigl\{G_{n;\alpha}\bigr\}_{n\in h\mathbb{N}}$ for which $G_{n;\alpha}$ is a balanced bipartite graph on $2n$ vertices, with minimum degree $\delta\bigl(G_{n;\alpha}\bigr) \geq \alpha n$, the sequence $\bigl\{G_{n;\alpha}\bigr\}$ satisfies the condition that if $p\geq Cn^{-(2h-1)/h^2}$ then, \whp, $G_{n;\alpha}\cup G_{n,n,p}$ contains a perfect $K_{h,h}$-tiling.
    
\end{theorem}

As a consequence of 
Remark~\ref{remark:divisibility}, the threshold function for the existence of a $K_{h,h}$-tiling of size at least $\lfloor n/h\rfloor$ in a randomly perturbed balanced bipartite graph, without the divisibility restriction is $n^{-(2h-1)/h^2}$.

In the event that we begin with an empty graph (i.e. $\alpha =0$), we note that a multiplicative polylog factor is required (Theorem~\ref{thm: covering}). The phenomenon of removing the polylog factor in the randomly perturbed setting when $\alpha$ is small is found throughout the literature. For example, it occurs in the tiling problem for general (i.e., not bipartite) randomly perturbed graphs \cite{balogh2019tilings, bottcher2023triangles}, as well as the case of finding Hamiltonian cycles \cite{bohman2003many}, and spanning trees \cite{krivelevich2017bounded}, and even in the hypergraph setting \cite{krivelevich2016cycles}.

Following a recent trend in the study of randomly perturbed graph theory such as in Han, Morris, and Treglown~\cite{HMT} and in B\"{o}ttcher, Parcyzk, Sgueglia, and Skokan~\cite{bottcher2023triangles}, we also consider the case in which $\alpha$ is bounded below by a nonzero constant. We show that, when $\alpha \geq 1/2- 1/n$, the threshold is no lower than $n^{-(h+1)/h}$.

\begin{theorem}\label{theorem: alpha large}
    There exists a sequence of graphs $\bigl\{G_{n;\alpha}^{*}\bigr\}_{n\in h\mathbb{N}}$ for which $G_{n;\alpha}^{*}$ is a balanced bipartite graph on $2n$ vertices, with minimum degree $\delta\bigl(G_{n;\alpha}^{*}\bigr) \geq \bigl(\frac{1}{n}\bigl\lceil\frac{1}{2}n\bigr\rceil-\frac{1}{n}\bigr)n$ and that satisfies the condition that if $p\leq cn^{-(h+1)/h}$ then, \whp, $G_{n;\alpha}^{*}\cup G_{n,n,p}$ does not contain a perfect $K_{h,h}$-tiling.
\end{theorem}

Zhao's~\cite{zhao2009bipartite} result says that $\alpha\geq \frac{1}{2}+\frac{3h-4}{2n}$ for $h\geq 2$ gives a perfect tiling with $p=0$ (by Hall's condition, $\alpha\geq 1/2$ gives the same for $h=1$).
The behavior of the threshold function close to $1/2$ is worth further investigation.
The case of $h=1$ is known, however, and the threshold function is $f(n)\cdot n^{-2}$, where $\alpha=1/2-f(n)$ and the relatively straightforward proof is in the first author's dissertation~\cite{dissertation}.

\subsubsection*{Organization}
In Section~\ref{sec:prelim} we discuss necessary preliminaries. In Section~\ref{sec: random graphs} we provide necessary background on random multipartite graphs in the spirit of results featured in Janson, {\L}uczak, and Ruci\'nski~\cite{JLR}. Theorem~\ref{thm: example} is proven in Section~\ref{sec: example}. Theorem~\ref{thm: main} is proven in Section~\ref{sec: proof main} for $h\geq 2$. The proof of Theorem~\ref{theorem: alpha large} appears in Section~\ref{section: alpha large}. In Section~\ref{sec:conclusion}, we provide concluding remarks. In Section~\ref{sec:appendix}, we prove Theorem~\ref{thm: main} for the special case of $h=1$.

\section{Preliminaries}
\label{sec:prelim}
\subsection{Notation}
For a graph $H$, we will use the notation $v_H = |V(H)|$ and $e_H = |E(H)|$ and $\chi(H)$ for the chromatic number of $H$. Let $\delta(H)$ denote the minimum degree of $H$, that is, the minimum size of the neighborhood of any vertex in $H$. Given a graph $G$, we use $N(x)$ to denote the neighborhood of a vertex $x\in V(G)$. We use $\deg(x)$ to denote $|N(x)|$ and we use $\deg(x,A)$ to denote $|N(x) \cap A|$ where $A \subseteq V(G)$.

The statement $0 < \alpha \ll \beta$ means that there exists an increasing function $f$ such that given $\beta$, we can choose $\alpha$ such that $\alpha \leq f(\beta)$. Colloquially, we just say ``$\alpha$ is small enough, compared to $\beta$''.

\subsection{Probabilistic methods} A basic and fundamental probabilistic inequality which we use later is Markov's inequality.

\begin{lemma}[Markov's inequality]\label{lemma: markov}
    If $X$ is a nonnegative random variable and $t>0$ is a real number, then $\prob\bigl(X \geq t\bigr) \leq \frac{\E [X]}{t}$. In particular, if $X$ is also integer-valued, then
    \begin{align*}
        \prob\bigl(X\geq 1\bigr) \leq \E[X] .
    \end{align*}
\end{lemma}


The version of the Chernoff bound that we use can be found in~\cite[Section 2.1]{JLR}.

\begin{lemma}[Chernoff bound]\label{lemma: chernoff}
    Let $X$ be either a binomial or hypergeometric random variable. Let $t\geq 0$. Then,
    \begin{align} 
        \prob\bigl(X \leq \E[X] -t\bigr) &\leq \exp\biggl( -\frac{t^2}{2\E[X]}\biggr), \label{eq:Chernoff1}
    \end{align}
\end{lemma}

Later we will use random slicing with respect to our Szemer\'edi partition. The following technical lemma establishes that \whp~for all vertices, the proportion of neighbors in a set does not change by much if one chooses a random subset. This is a common argument, found in e.g. Balogh, Treglown, and Wagner~\cite{balogh2019tilings}.
\begin{lemma}\label{lemma: random slicing min deg}
Let $0< \alpha, \beta \leq 1$ and $0<\beta'\leq\beta$. Given a bipartite graph $G = (A,B;E)$, where $|A|=\alpha n$ and $|B|=\beta n$ if $B' \subseteq B$ is chosen uniformly at random from all sets of size $\beta'n$, then for every $a\in A$, $\E[\deg(a,B')]=\deg(a,B)\frac{\beta'}{\beta}$. 
Furthermore, \whp~it is the case that for each $a\in A$, we have $\deg(a,B') \geq \deg(a,B)\frac{\beta'}{\beta}-\sqrt{2n}\,\ln n$.
\end{lemma}
\begin{proof}
    The variable $X_a=\deg(a,B')$ is hypergeometric with parameter $\lambda=\E[X_a]=\deg(a,B)(\beta'n)/(\beta n)$ by~\cite[Theorem 2.10]{JLR}. inequality. Indeed, setting $t = \sqrt{2n}\,\ln n$ and using the union bound and then Chernoff's inequality~\eqref{eq:Chernoff1},
    \begin{align*}
        \prob\bigl(\exists a\in A : X_a \leq \E[X_a] - t\bigr) \leq \sum_{a \in A}\exp{\biggl(\frac{-t^2}{2\E[X_a]}\biggr)} \leq |A|\exp\biggl(-\frac{n (\ln n)^2}{\beta' n}\biggr).
    \end{align*}
    This goes to zero using the fact that $|A|\leq n$ and $\beta',\alpha<1$.
\end{proof}

\subsection{Tilings in bipartite graphs}
In proving Theorem~\ref{thm: main}, we will need to tile an auxiliary bipartite graph by copies of $K_{1,t}$, where $t$ is a large constant to be determined. Our method for accomplishing this is done by Theorem~\ref{thm: star tiling}. In order to state the theorem, we need some preliminaries.

\begin{definition}[Koml\'os~\cite{komlostilingturan}]\label{def:critical chromatic}
The critical chromatic number of a graph $F$, denoted $\chicr(F)$, is defined as 
\begin{align*}
    \chicr(F) = \frac{\bigl(\chi(F)-1\bigr)v_F}{v_F - \sigma(F)}
\end{align*}
where $\sigma(F)$ is the size of the smallest color class over all proper $\chi(F)$-colorings of $F$. 
\end{definition}


\begin{theorem}[Bush and Zhao~\cite{Bush}] [Theorem 1.5]\label{thm: star tiling}
Let $F$ be a bipartite graph. 
There exists $n_0$ and $c(F) < 8v_{F}^{2}$ such that every balanced bipartite graph $G$ on $n \geq n_0$ vertices contains an $F$-tiling covering all but at most $c(F)$ vertices if $\delta(G) \geq (1- 1/\chicr(F))n$.
\end{theorem}

Applying Theorem~\ref{thm: star tiling} to the star graph, $F=K_{1,t}$, which has $\chicr\bigl(K_{1,t}\bigr)=(t+1)/t$, and removing some copies of $K_{1,t}$ to ensure that there are an equal number of stars rooted in each side of the bipartition. We obtain the following Corollary: 

\begin{corollary}\label{cor:bipartite stars}
Let $t$ be a positive integer. There exists $n_0$ and $c(t) < 8(t+1)^{2}\frac{t}{t-1}$ such that every balanced bipartite graph $G$ on $n \geq n_0$ vertices contains a balanced $K_{1,t}$-tiling covering all but at most $c(t)$ vertices in each part if $\delta(G) \geq n/(t+1)$.
\end{corollary}

We will also make use of a bipartite version of the degree form of Szemer\'{e}di's Regularity~\cite{szemreglem} (our Lemma~\ref{lemma: multipartite regularity} below). Before stating Lemma~\ref{lemma: multipartite regularity}, we first make the necessary preparations. 

\subsection{Epsilon-regular pairs}
For disjoint vertex sets $A$ and $B$, let $e(A,B)$ denote the number of edges with an endpoint in $A$ and an endpoint in $B$. Some definitions in papers using Szemer\'edi's regularity lemma vary slightly, we will follow the definitions in~\cite{balogh2019tilings}.

\begin{definition}
For disjoint vertex sets $A$ and $B$, the \textit{density} between $A$ and $B$ is
\begin{align*}
    d_G(A,B):= \frac{e(A,B)}{|A||B|}.
\end{align*}

Given $\epsilon>0$, we say that a pair of disjoint vertex sets $(A,B)$ is \textit{$\epsilon$-regular} if for all sets $X \subseteq A$, $Y\subseteq B$ with $|X| \geq \epsilon |A|$ and $|Y| \geq \epsilon |B|$ we have 
\begin{align*}
    |d_{G}(A,B) - d_{G}(X,Y) | < \epsilon.
\end{align*}
Given $d \in [0,1]$, we say that a pair of disjoint vertex sets $(A,B)$ is \textit{$(\epsilon,d)$-super-regular} if the following two properties hold:
\begin{itemize}
    \item for all sets $X \subseteq A$, $Y\subseteq B$ such that $|X| \geq \epsilon |A|$ and $|Y| \geq \epsilon |B|$, we have $d_{G}(X,Y) >d$;
    \item for all $a \in A$ and $b \in B$, we have $\deg_{G}(a) > d|B|$ and $\deg_{G}(b) >d|A|$.
\end{itemize}

\end{definition}



We will use several well-known properties of $\epsilon$-regular-pairs. The proofs can be found, e.g., in Treglown~\cite{treglown2007regularity} and are omitted here.


\begin{lemma}[Slicing Lemma, \cite{treglown2007regularity}, Lemma 1.6]\label{lemma: deterministic slicing}
Let $0 < \epsilon<\alpha $ and $\epsilon':= \max \{\epsilon/\alpha, 2\epsilon\}$. Let $(A,B)$ be an $\epsilon$-regular pair with density $d$. Suppose $A' \subseteq A$ such that $|A'| \geq \alpha|A|$, and $B'\subseteq B$ such that $|B'| \geq \alpha |B|$. Then $(A',B')$ is an $\epsilon'$-regular pair with density $d'$ where $|d'-d| < \epsilon$.
\end{lemma}

We will regularly make use of the following immediate consequence of Lemma~\ref{lemma: deterministic slicing}. 
\begin{corollary}\label{cor: slicing}
Let $(A,B)$ is an $(\epsilon,d)$-super-regular pair with $|A|=|B|=L$. If at most $\epsilon_1 L$ vertices are removed from each of $A$ and $B$ to obtain $A' \subseteq A$ and $B'\subseteq B$, then $(A',B')$ is $(\epsilon',d-\epsilon_1)$-super-regular with $\epsilon' = \max \{ \epsilon/\epsilon_1, 2\epsilon\}$.
\end{corollary}

We will also make use of Lemma~\ref{lemma:random slicing}, which states that randomly slicing a super-regular pair into a large constant number of parts ensures that each pair of parts is super-regular with relaxed parameters. The result is a ``whp'' result and this holds as long as $m=|A|=|B|$ goes to infinity. 

\begin{lemma}[Random Slicing, \cite{csaba2012approximate}, Proposition 10]\label{lemma:random slicing}
Let $(A, B)$ be an $(\epsilon,d)$-super-regular pair with density $d(A,B)$ and let $k$ be a positive integer. Assume that $|A| = |B| = m$, and $k\mid m$. Partition $A$ and $B$ into $k$ equally-sized subsets uniformly at random:
$A = A_1 \sqcup A_2 \sqcup \dots \sqcup A_k$ and $B = B_1 \sqcup B_2 \sqcup \dots \sqcup B_k$. Then whp for every $i,j\in [k]$, the pair $(A_i, B_j)$ is $(\epsilon', d')$-super-regular with density at least $d(A,B)-\epsilon$, where $\epsilon' \leq 2\epsilon$ and $d - \epsilon \leq d'$.
\end{lemma}

Lemma~\ref{lemma: super-regularity} shows that every regular pair contains a large super-regular pair. 

\begin{lemma}[\cite{treglown2007regularity}, Lemma 1.8]\label{lemma: super-regularity}
If $(A,B)$ is an $\epsilon$-regular pair with density $d$ in a graph $G$ (where $0< \epsilon < 1/3)$, then there exists $A' \subseteq A$ and $B' \subseteq B$ with $|A'| \geq (1-\epsilon) |A|$ and $|B'| \geq (1-\epsilon)|B|$, such that $(A',B')$ is a $(2\epsilon,d-3\epsilon)$-super-regular pair. 
\end{lemma}



The Intersection Property, Lemma~\ref{lemma: intersection property}, and Corollary~\ref{cor: bipartite copies}, which immediately follows, establishes that $\epsilon$-regular pairs contain many copies of complete bipartite graphs.

\begin{lemma}[Intersection Property,~\cite{komlos1996regularity}, Fact 1.4] \label{lemma: intersection property}
Let $0< \epsilon<d$ and $h\geq 1$ be an integer. Let $(A,B)$ be $\epsilon$-regular with density $d$. If $Y \subseteq B$ and $(d-\epsilon)^{h-1}|Y| \geq \epsilon|B|$, then
\begin{align*}
    \Biggl| \Biggl\{(a_1, a_2, \dots,a_h) \in \binom{A}{h} : \biggl|Y \cap \biggl( \bigcap_{i=1}^{h} N(a_i)\biggr)\biggr| < (d-\epsilon)^{h}|Y| \Biggr\} \Biggr| \leq h\epsilon |A|^{h}.
\end{align*}
\end{lemma}

\begin{corollary}\label{cor: bipartite copies}    
Let $k,k'$ be positive integers. If $(A,B)$ is an $\epsilon$-regular pair with density $d$ such that 
\begin{itemize}
    \item $(d-\epsilon)^{k-1} \geq \epsilon$,
    \item $k\epsilon |A|^{k} < (|A|)_{k}$,
    \item $(d-\epsilon)^{k}|B|\geq k'$.
\end{itemize}
Then $(A,B)$ contains a copy of $K_{k,k'}$ with $k$ vertices in $A$ and $k'$ vertices in $B$. In particular, if $\epsilon <1/k$ and $d \geq \epsilon^{1/(k-1)} + \epsilon$ and $|A|,|B|$ are large enough, then $(A,B)$ has a copy of $K_{k,k'}$.
\end{corollary}

Corollary~\ref{cor: (v,A,B)} establishes that, given a graph $G$ containing an $(\epsilon,d)$-super-regular pair $(A,B)$ and a vertex $v$ outside of $A\cup B$ such that $v$ has a large neighborhood in $A$, then there exists a copy of $K_{h,h}$ where all vertices except $v$ are in $A\cup B$. 

\begin{corollary}
    \label{cor: (v,A,B)}
    Let $0<\epsilon \ll d \ll d'$ and $h\geq 1$ be a positive integer.
    Let $(A,B)$ be a $(\epsilon, d)$-super-regular pair  with $(1-\epsilon)n \leq |A|,|B| \leq n$ and let $v\not\in A\cup B$. If $\deg(v, A) \geq d'|A|$, then there exists a copy of $K_{h,h}$ which contains $v$, $h$ vertices in $A$, and $h-1$ vertices in $B$.
\end{corollary}

\subsection{Szemer\'{e}di's Regularity Lemma}
Finally, we state a version of the Regularity Lemma that can be derived from the original. 
We will make use of a multipartite version of the degree form of the regularity lemma (see Martin, Mycroft, and Skokan~\cite[Theorem 2.8]{martin2017asymptotic}) for the statement of the degree form). We will refer to Lemma~\ref{lemma: multipartite regularity} as ``the Regularity Lemma'' throughout this paper. For this paper, we only use the case that $r=2$. 

\begin{lemma}[Regularity Lemma~\cite{martin2017asymptotic}]\label{lemma: multipartite regularity}
    For every integer $r \geq 2$ and every $\epsilon>0$, there is an $M = M(r,\epsilon)$ such that if $G = (V_1, V_2, \dots, V_r; E)$ is a balanced $r$-partite graph on $rn$ vertices and $d \in [0,1]$ is any real number, then there exists integers $\ell$ and $L$, a spanning subgraph $G' = (V_1, \dots, V_r;E')$ and for each $i=1,\dots,r$ a partition of $V_i$ into clusters $V_{i}^{0}, V_{i}^{1}, \dots, V_{i}^{\ell}$ with the following properties:
    \begin{enumerate}
        \item $\ell \leq M$,
        \item $|V_i^{0}| \leq \epsilon n$ for all $i \in [r]$,
        \item $|V_{i}^{j}| = L \leq \epsilon n$ for $i\in [r]$ and $j \in [\ell]$,
        \item $\deg_{G'}(v,V_{i'}) > \deg_{G}(v,V_{i'}) - (d+\epsilon)n$ for all $v \in V_i$, $i \not = i'$\label{bp: degrees} and 
        \item all pairs $(V_{i}^{j}, V_{i'}^{j'})$ with $i \not = i'$, $j,j' \in [\ell]$ are $\epsilon$-regular with density exceeding $d$ or $0$.    
    \end{enumerate}
\end{lemma}
    
After applying Szemer\'{e}di's Regularity Lemma (Lemma~\ref{lemma: multipartite regularity}) to the deterministic graph $G$, we will define the \emph{Szemer\'{e}di graph} $\GSz$ obtained by taking its vertices as the vertex classes $V_i$ of $G$ with edges $\bigl\{V_i,V_j\bigr\}$ whenever $\bigl(V_i, V_j\bigr)$ forms an $\epsilon$-regular pair with density at least $d$. The Szemer\'{e}di graph partially inherits the minimum degree of $G$. Lemma~\ref{lemma: reduced graph degree} makes this statement precise. 

\begin{lemma}\label{lemma: reduced graph degree}
Let  $0 < \epsilon \ll d \ll \alpha$. If $G=(A,B; E)$ is a balanced bipartite graph  on $2n$ vertices with $\delta(G) \geq \alpha n$ then any corresponding balanced Szemer\'{e}di graph, $\GSz$, on $2\ell$ vertices has  $\delta(\GSz) \geq \bigl(\alpha - d - 2\epsilon\bigr) \ell$. 
\end{lemma}



Given a bounded degree subgraph $Q$ of the Szemer\'{e}di graph (usually spanning), there is a way to slice away a few vertices from each cluster so that for the resulting graph, every pair that was regular is still regular with a relaxed parameter and every pair in $E(Q)$ itself is super-regular.
\begin{lemma}[Super-regularization]\label{lemma: super-regularization}
    Let $0<d\ll 1$ and $\Delta$ and $\epsilon$ be such that $\Delta \cdot \epsilon<1/2$ and $(2\Delta+1)\epsilon<d$. There is an $L_0$ such that for all $L\geq L_0$ the following holds: 
    
    Let $\GSz$ be a Szemer\'edi graph with clusters of size $L$ such that every pair is $\epsilon$-regular with density at least $d$ for pairs in $E(\GSz)$ and with density zero for pairs not in $E(\GSz)$. Let $Q$ be a subgraph of $\GSz$ with maximum degree at most $\Delta$.
    Let $L'=(1-d')L$ be an integer for some $d'$ that satisfies $d'<\Delta \epsilon<d'+(d')^2$. 
    For every $C\in V(\GSz)$ there is a $C'\subset C$ of size exactly $L'=(1-d')L$ such that
    \begin{enumerate}
        \item For every $\bigl(C_1,C_2\bigr)\in E(\GSz)$, the pair $\bigl(C'_1,C'_2\bigr)$ is $2\epsilon$-regular with density at least $d-\epsilon$. \label{it:super-reg:reg}
        \item For every $\bigl(C_1,C_2\bigr)\in E(Q)$, the pair $\bigl(C'_1,C'_2\bigr)$ is $\bigl(2\epsilon,d'\bigr)$-super-regular.\label{it:super-reg:super}
        \end{enumerate}
    \label{lemma:super-regularization}
\end{lemma}

We will make use of the Blow-up Lemma which allows us to embed a graph $H$ into another graph $G$.
\begin{theorem}[Blow-up Lemma, Koml\'os, S\'ark\"ozy, and Szemer\'edi~\cite{komlosBlowUp}]\label{lemma: blow-up}
    Given a graph $R$ of order $r$ and positive parameters $d, \Delta$, there exists an $\epsilon>0$ such that the following holds: Let $N$ be an arbitrary positive integer, and let us replace the vertices of $R$ with pairwise disjoint $N$-sets $V_1, V_2, \dots, V_r$. We construct two graphs with the same vertex set $V = \cup V_i$. The graph $R(N)$ is obtained by replacing all edges of $R$ with copies of the complete bipartite graph $K_{N,N}$ and a sparser graph $G$ is constructed by replacing the edges of $R$ with some $(\epsilon,d)$-super-regular pairs. If a graph $H$ with maximum degree $\Delta(H) \leq \Delta$ can be embedded into $R(N)$, then it can be embedded into $G$.
\end{theorem}

\section{Random multipartite graphs}\label{sec: random graphs}

Theorem~\ref{thm: covering} implies that the polylog factor in the threshold for a perfect bipartite tiling is necessary as it is in the general random graph case in Johansson, Kahn, and Vu~\cite{JKV}. However, in Theorem~\ref{thm:threshold} we prove that, in the perturbed case the threshold does not have this polylog factor. To that end, we state a special case for the bipartite setting.

\begin{theorem}[Gerke and McDowell~\cite{GerkeMcDowell}]\label{thm: covering} 
   \[\lim_{n\to\infty } \mathbb{P}(G_{n,n,p} \text{ contains a perfect $K_{h,h}$-tiling}) =\begin{cases} 
      1, & \text{if } p =\omega\bigl((\log n)^{1/h^2}n^{-(2h-1)/h^2}\bigr); \\
      0,  & \text{if } p =o\bigl((\log n)^{1/h^2}n^{-(2h-1)/h^2}\bigr).
   \end{cases}
\]
\end{theorem}

Lemma~\ref{lemma: partial covering} below will be useful for finding many copies of $K_{h,h}$ in sufficiently large, dense subgraphs. The following theorem is proved along the same lines as \cite[Theorem~4.9]{JLR} and was proved originally by Ruci\'nski~\cite{RucinskiExtension} in a far more general setting in which the graph to be tiled need not be $K_{h,h}$ but can be any ``strictly balanced" graph and $G_{n,n,p}$ is replaced by $G_{n,p}$.

\begin{lemma}[Partial bipartite tiling]\label{lemma: partial covering}
Let $\epsilon >0$, and $h\geq 1$ be a positive integer. Let $F(\epsilon,h)$ be the property that a graph contains a $K_{h,h}$-tiling covering all but at most $\epsilon n$ vertices. There exist $c=c(\epsilon,h)$ and $C=C(\epsilon,h)$ such that
\[\lim_{n\to\infty } \mathbb{P}\bigl(G_{n,n,p} \in F(\epsilon,h)\bigr) =\begin{cases} 
    1,  & \text{if }p \geq C n^{-(2h-1)/h^2};\\
      0& \text{if } p \leq cn^{-(2h-1)/h^2}.
   \end{cases}
\]
\end{lemma}
Note that by Lemma~\ref{lemma: partial covering}, in order to get only a partial $K_{h,h}$-tiling we don't need the polylog factor.

\section{Extremal example}\label{sec: example}
We prove Theorem~\ref{thm: example}. Let $h\geq 1$ be a positive integer. Let $\alpha + \beta =1$, with $0<\alpha < 1/(2h)$ and $\beta > (2h-1)\alpha$. Then, let $G_{n;\alpha}^{*}$ have vertex classes $A = A_1 \sqcup A_2$ and $B =B_1 \sqcup B_2$ with $\bigl|A_1\bigr| = \bigl|B_1\bigr| = \alpha n$ and $\bigl|A_2\bigr| = \bigl|B_2\bigr| = \beta n$. The graph $G_{n;\alpha}^{*}$ has all edges in each of the pairs $\bigl(A_1, B_1\bigr)$, $\bigl(A_1, B_2\bigr)$, $\bigl(A_2,B_1\bigr)$ and no edges in the pair $\bigl(A_2, B_2\bigr)$. 
For a contradiction, suppose that $G_{n;\alpha}^{*}\cup G_{n,n,p}$ contains a perfect $K_{h,h}$-tiling. 
The number of copies of $K_{h,h}$ using at least one vertex in $A_1 \cup B_1$ is at most $2\alpha n$.

Therefore, this tiling covers at most $2\alpha n(2h-1)$ vertices in $A_2\cup B_2$. 
Note that $2\beta n-2\alpha n(2h-1)=2n(1-2h\alpha)>0$ as long as $\alpha<1/(2h)$. 
However, applying Lemma~\ref{lemma: partial covering} to $G_{n,n,p}[A_2 \cup B_2] \cong G_{\beta n, \beta n, p}$, with $\epsilon = \frac{\alpha (2h-1)n}{\beta n} = \frac{\alpha (2h-1)}{1-\alpha}$ gives that such a partial $K_{h,h}$-tiling does not exist \whp~for $p=O\bigl(n^{-(2h-1)/h^2}\bigr)$.

\section{Proof Theorem~\ref{thm: main}}\label{sec: proof main}

Let $h\geq 2$. We begin with the usual sequence of constants:
\begin{align*}
0<\eta \ll \epsilon \ll \epsilon_1 \ll \epsilon_2\ll d \ll d' \ll \alpha < 1/(2h)   
\end{align*}
and let $t=\lceil 2\alpha^{-1} \rceil$. We apply the bipartite version of the degree form of the Regularity Lemma (Lemma~\ref{lemma: multipartite regularity}) to  $G = G_{n;\alpha}$ with $\ell_0 = \epsilon^{-1}$. The spanning subgraph $G'$ of $G$ obtained from the Regularity Lemma respects the original bipartition of $G$ and so we denote the clusters belonging to $A$ and $B$ by $V_i^{A}$ and $V_{j}^{B}$ respectively, for $i,j=0,1,\dots, \ell$, where $\ell_0\leq\ell\leq M$. Let $\mathcal{A}$ denote the set of clusters in $A$ (excluding the leftover set, $V_{0}^{A}$). Similarly, let $\mathcal{B}$ denote the set of clusters in $B$ (again, excluding the leftover set, $V_{0}^{B}$). We shall refer to the vertices of the Szemer\'{e}di graph, $\GSz$, as clusters.

By Lemma~\ref{lemma: reduced graph degree}, $\delta(\GSz) \geq (\alpha - d/2 -3\epsilon)\ell$. Since
\begin{align*}
    \chicr(K_{1,t}) = \cfrac{(2-1)(t+1)}{t+1 -1} = \cfrac{t+1}{t}  ,
\end{align*}
our choice of $t$ gives $(\alpha - d/2 - 3\epsilon)\ell \geq (\alpha/2) \ell \geq \ell/(t+1) $. From Corollary~\ref{cor:bipartite stars}, we find a partial $K_{1,t}$-tiling $\mathcal{K}$ of $\GSz$ covering all but at most $c_2< 8(t+1)^2\frac{t}{t-1}$ clusters. We discard these at most $c_2$ clusters by adding them to the respective leftover set. Consequently, since $t\geq 2$, the number of stars centered at $\mathcal{A}$ is equal to the number of stars with centers in $\mathcal{B}$. For ease of notation, let $\ell$ be redefined to denote the number of clusters remaining.


We follow the notation in \cite{balogh2019tilings}. For an arbitrary ordering of the stars in the $K_{1,h}$-tiling, $j=1,2,\ldots,k=2\ell/(t+1)$ and, given such a $j$, an ordering of the clusters of the star, $i=0,1,\ldots,t$ (where the $0^{\rm th}$ cluster is the center of the star), let $V_{i,j}^{X}$ denote the $(j,i)^{\rm th}$ cluster where $X \in \{A,B\}$ designates whether the cluster is in $A$ or $B$. We will leave off the superscript when the context is clear or we mean to refer to an arbitrary part of the bipartition. That is, $V_{0,j}$ denotes the center of the $j^{\rm th}$ star and we say that this star is \textit{centered at $V_{0,j}$}.

We have that the number of leftover vertices is now at most
\begin{align*}
    \left|V_0\right| \leq  \epsilon n + c_2 L\leq (c_2 + 1) \epsilon n < 9 \biggl(\frac{(t+1)^2t}{t-1}\biggr)\epsilon n.
\end{align*}

Next, we move exactly $\lfloor t \epsilon L\rfloor$ vertices from each $V_{i,j}$ to the leftover set in order to get that, by Lemma~\ref{lemma: super-regularization}, the resulting cluster pairs $(V_{0,j} , V_{i,j})$ are $(2\epsilon,d/2)$-super-regular for each $j\in [k]$ and all $i\in [t]$. 
After moving these vertices to the leftover set, the size of each leftover set can be bounded above:
\begin{align*}
    \left|V_0\right| &\leq 9 \Bigl(\frac{(t+1)^2t}{t-1}\Bigr)\epsilon n + \ell(t \epsilon L) \\
    &\leq 45t^2(\epsilon n)\\
    &\leq \frac{180\epsilon}{\alpha^2} n .
\end{align*}

 Our strategy for tiling the leftover vertices is to pair each leftover $v\in V_0^A\cup V_0^B$ with a cluster $V_{i,j} \subseteq \mathcal{S} \in \mathcal{K}$ and to use the vertices in that cluster to form a copy of $K_{1,h}$ with $v$, then we choose the remaining $h-1$ vertices from $V_{i',j'}$, where $(V_{i,j}, V_{i',j'})$ are $(2\epsilon,d/2)$-super-regular. 
 In order to make sure that we don't use a cluster too many times we seek an assignment of the vertices of $V_{0}^{A}$ to the clusters of $\mathcal{B}$. A similar argument will hold for the vertices of $V_{0}^{B}$ assigned to the clusters of $\mathcal{A}$. 
 For ease of notation, let $L$ be redefined to denote the number of vertices remaining in each cluster.
 
 Let $J$ be an auxiliary bipartite graph with one side being the vertices of $V_{0}^{A}$ and the other being $\mathcal{B}$. 
 We include the edge $\{v,V_{i}^{B}\}$ whenever $\deg_{G'}(v,V_{i}^{B})~\geq d'L$. 
 
 We claim that $J=(V_{0}^{A}, \mathcal{B})$ has 
 \begin{align}
     \deg_{J}(v) \geq (\alpha -3d') \ell \qquad \mbox{for each $v \in V_{0}^{A}$}.
     \label{eq:Qdeg}
    \end{align}
 Suppose this does not occur for some $v$, then
 \begin{align*}
     \deg_{G'}(v) &\leq \sum_{V_{i,j}\in\mathcal{B}}\deg_{V_{i,j}}(v) + \bigl|V_0^B\bigr| \\
     &\leq (\alpha-3d')\ell L + \ell d' L + \bigl|V_0^B\bigr| \\
     &\leq (\alpha-2d')n + \frac{180\epsilon}{\alpha^2}n \\
     &\leq (\alpha-d')n .
 \end{align*}

On the other hand, from the Regularity Lemma (Lemma (\ref{lemma: multipartite regularity})(iv)),
 \begin{align*}
     \deg_{G'}(v) &\geq \deg_{G}(v) - (d+\epsilon)n \\
     &\geq \alpha n - (d+\epsilon)n \\
     & \geq (\alpha - 2d)n,
 \end{align*}
 which is a contradiction because $d\ll d'$. 

 We use a greedy algorithm to assign each of the leftover vertices to a cluster while making sure that no cluster is assigned to too many vertices.


\begin{fact}\label{claim:bipartite assignment}
    There exists an assignment of the vertices of $V_0$ to the clusters in $\GSz$ such that each cluster is assigned to at most $|V_{0}^{A}|/ (\alpha - 3d')\ell$ many vertices.
\end{fact}

 By Fact~\ref{claim:bipartite assignment}, we have that, the number of vertices from $V_{0}^{A}$ that are assigned to a cluster from $\mathcal{B}$ is at most 
 
 \begin{align*}
    \frac{|V_{0}^{A}|}{(\alpha - 3d')\ell} 
        = \frac{|V_{0}^{A}|}{\alpha - 3d'}\cdot \frac{L}{n-\bigl|V_{0}^{B}\bigr|}
     \leq \frac{180\epsilon/\alpha^2}{(\alpha - 3d')(1-180\epsilon/\alpha^2)n} L
     \leq \frac{\epsilon_1}{h}L .
 \end{align*}
 
 We now use Corollary~\ref{cor: (v,A,B)}, in order to find vertex-disjoint copies of $K_{h,h}$, one for each leftover vertex $v$. That is, take $v\in V_{0}^{A}$ and find $(V_{i,j}^{B}, V_{i',j}^{A})$ such that $V_{i,j}^{B}$ and $V_{i',j}^{A}$ are neighbors in the star tiling $\mathcal{K}$ (hence either $i=0$ or $i'=0$), and $v$ is assigned to  $V_{i,j}^{B}$. We find $h$ vertices in $V_{i,j}^{B}$ and $h-1$ vertices in $V_{i',j'}^{A}$ such that together with $v$, they form a copy of $K_{h,h}$. We must verify that we can find these vertex-disjoint copies of $K_{h,h}$.

To that end, there are at most $\bigl\lfloor (\epsilon_1/h)L\bigr\rfloor$ many vertices assigned to each cluster of $\mathcal{B}$. We remove a vertex from a cluster $V_{i,j}^{B}$ if it is either assigned to a vertex of $v \in V_{0}$ or if it is in a cluster $V_{i',j}^{A}$ as described above. Therefore at most $(1+t)\epsilon_1 L$ vertices are removed from any cluster. Recall that we had a similar assignment of vertices $V_{0}^{B}$ to members of $\mathcal{A}$. In sum, at most $2(1+t)\epsilon_1 L$ vertices were removed from each cluster.

As a consequence of so few vertices being removed, we have that by the Slicing Lemma~(Lemma \ref{cor: slicing}), the vertices that remain of each pair $\bigl(V_{i',j}^{A}, V_{i,j}^{B}\bigr)$ is $(2\epsilon_1,d/3)$-super-regular, provided $\epsilon_1 \ll d$ and $d$ is sufficiently small. 
Although we have possibly removed some vertices from each cluster, for ease of notation, we do not modify the names of the cluster. 
At this stage in the proof, we have that number of vertices in $V_{i,j}$ is at most $L$ and at least $(1-2(t+1)\epsilon_1)L$.



 Now, we will argue that we can remove a small number of vertices in order to make the size of what remains of each of the clusters divisible by $h$. 
 We accomplish this by locating pairs of clusters from the same side of the bipartition, say $(V_{i,j}^{A},V_{i',j'}^{A})$, whose sizes are not divisible by $h$, and locating a cluster $V_{x,j}^{B}$, thus $(V_{i,j}^{A}, V_{x,j}^{B})$ is $(2\epsilon_1,d/3)$-super-regular.  Then, we will find a copy of $K_{h,h}$ which uses $1$ vertex from $V_{i,j}^{A}$ and $h$ vertices of $V_{x,j}^{B}$ and $h-1$ vertices of $V_{i',j'}^{A}$.

To this end, first we arbitrarily order all of the clusters $V_{i,j}^{A}$ for which $h \nmid |V_{i,j}^{A}|$. We will use an additional subscript to denote this ordering (but will drop this subscript when the context is unambiguous). Note that these clusters must appear in pairs since $h \mid n$. Consider a pair in this ordering $(V_{i,j,q}^{A}, V_{i',j',q+1}^{A})$. At each stage we will have removed at most $h-1$ vertices from $V_{i,j,q}^{A}$ and at most $(h-1)(h-1)$ vertices from $V_{i',j',q+1}^{A}$.

We have two cases, either the divisibility issue occurs within the leaves of the same star or the divisibility issue occurs within clusters belonging to disjoint stars. In the latter case, we handle the issue using random edges. 

Fix a pair $(V_{i,j,q}^{A}, V_{i',j',q+1}^{A})$ as above.

In the first case, suppose $j = j'$ and $i\neq i'$. Recall, each of $(V_{i,j}^{A}, V_{0,j}^{B})$ and $(V_{i',j}^{A}, V_{0,j}^{B})$ are $(2\epsilon_1,d/3)$-super-regular. We choose a vertex $v \in V_{i,j}^{A}$ arbitrarily and apply Corollary~\ref{cor: (v,A,B)} to find a copy of $K_{h,h}$ covering $v$ and $h$ vertices of $V_{0,j}^{B}$, and $h-1$ vertices of $V_{i',j}^{A}$.


In the second case, suppose that $j \neq j'$. Since $(V_{i,j}^{A},V_{x,j}^{B})$ is $(2\epsilon_1, d/3)$-super-regular, we can choose $v \in V_{i,j}^{A}$ arbitrarily and note that $\deg(v,V_{x,j}^{B}) \geq \frac{d}{3}(1-\epsilon_1)|V_{x,j}^{B}|$. Now, by Lemma~\ref{lemma: partial covering}, find a copy of $K_{h-1,h}$ with $h$ vertices in $N_{V_{x,j}^{B}}(v)$ and $h-1$ vertices in $V_{i,j'}^{A}$.

In either case, the cluster $V_{x,j}^{B}$ loses at most $h(h-1)$ many vertices. Since there are at most $(h-1)\ell$ iterations, then each cluster loses at most $h(h-1) \cdot (h-1)\ell$ many vertices. We apply the same process to the clusters in $\mathcal{B}$ by exchanging the roles of $\mathcal{A}$ and $\mathcal{B}$. 

By the end of this process, each cluster has lost at most $2h^3 \ell=O(1)$ vertices. Now each $(V_{0,j}, V_{i,j})$ is $(3\epsilon_1, d/4)$-super-regular, provided $\epsilon_1\ll d$ and $d$ is sufficiently small. Moreover, every cluster is now divisible by $h$.


We will now ensure that all of the clusters that form centers of stars, that is, clusters labeled $V_{0,j}$, can be made into the same size, which is divisible by $ht$. To this end, set $L_1= ht\bigl\lfloor (1-3\epsilon_1h)L/ht \bigr\rfloor$. For each $V_{0,j}$, Select $V_{i,j}$ arbitrarily, $i \in [t]$, and apply Lemma~\ref{cor: bipartite copies} to find at most $3\epsilon_1L$ vertex-disjoint copies of $K_{h,h}$ each with $h$ vertices of $V_{0,j}$ and $h$ vertices of $V_{i,j}$. By the Slicing Lemma (Corollary~\ref{cor: slicing}), $(V_{0,j}, V_{i,j})$ is $(6\epsilon_1, d/5)$-super-regular for each $i \in [t]$ and for each $j \in [k]$.


We will now make all of the clusters that form leaves to be the same size $L_1$ by pairing up leaves of size greater than $L_1$ and making use of random edges.
Because $\mathcal{A}$ and $\mathcal{B}$ have the same number of clusters, then for every $V_{i,j}^{A}$ with size exceeding $L_1$, we can find $V_{i',j'}^{B}$ also with size exceeding $L_1$.
While $\bigl|V_{i,j}\bigr| -L_1 >0$ we use Lemma~$\ref{lemma: partial covering}$ to remove copies of $K_{h,h}$ greedily.
Note that throughout this process, the number of remaining vertices in $V_{i,j}$ will be divisible by $h$, since $\bigl|V_{i,j}\bigr|$ was made to be divisible by $h$ before the beginning of this process.
By the Slicing Lemma (Corollary~\ref{cor: slicing}), each $(V_{0,j}, V_{i,j})$ is now $(12\epsilon_1, d/6)$-super-regular. 
Note also that the clusters  are all of size $L_1$, which is divisible by $ht$.


Our goal is now to partition each of the centers $V_{0,j}$ into $t$ equally sized parts, then assign each part to one of the $t$ leaves.
We arbitrarily match the leaves between disjoint stars and find a large partial tiling between each pair using random edges.
This partial tiling will leave $L_1/t$ uncovered vertices in each leaf.
Finally, we will match these uncovered vertices from $V_{i,j}$ to a corresponding part from the center $V_{0,j}$.

To that end, for each $j \in [k]$, partition $V_{0,j}$ uniformly at random into $t$ equally sized pieces $V_{0,j}= \bigsqcup_{i=1}^{t} T_{i,j}$ such that $\bigl|T_{i,j}\bigr| = L_1/t$. Also, we randomly partition $V_{i,j}$ into two pieces, one of which we'll define to be $S_{i,j}'$ and has size $L_1/s$ and $V_{i,j}'$ with size $(s-1)L_1/s$ which satisfies $1/t = 1/s + \eta(s-1)/s$. Hence $1/s<1/t$. Recall that $\eta \ll \epsilon$.

Since there are the same number of leaves in $\mathcal{A}$ as in $\mathcal{B}$, we find a perfect matching among the leaves (in $\GSz$) and find a large partial tiling between each pair of leaves.
By our choice of $\eta$ we apply Lemma \ref{lemma: partial covering} to the pair $(V_{i_1, j_1}', V_{i_2, j_2}')$ to find a partial tiling of all but at most $\eta(s-1)L_1/s$  vertices.
We denote the uncovered vertices by $S_{i_1,j_1}''\subset V_{i_1,j_1}'$ and $S_{i_2,j_2}''\subset V_{i_2,j_2}'$. 
For each $j \in [t]$ and for each $i \in [k]$, let $S_{i,j} = S_{i,j}' \sqcup S_{i,j}''$. 
For $L_1$ (hence $L_1/(ht)$) sufficiently large, there exists $\eta \ll \epsilon$ such that $L_1/s$ is divisible by $h$.
We have
 \begin{align*}
     \bigl|S_{i,j}\bigr| = \bigl|S_{i,j}'\bigr| + \bigl|S_{i,j}''\bigr|=\frac{1}{s}\,L_1 + \eta\, \frac{s-1}{s}\,L_1 = \frac{L_1}{t}.
 \end{align*}

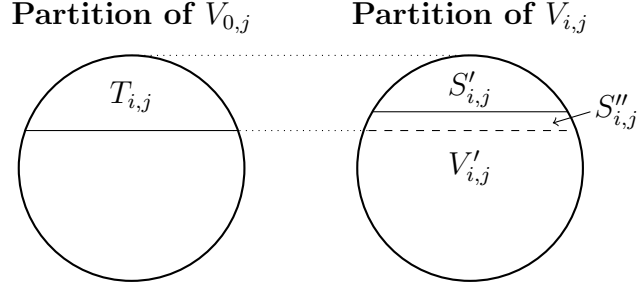
\begin{figure}
\centering
\begin{tikzpicture}[scale=.5]
\useasboundingbox (-12,-3) rectangle (3,5);
\draw[thick] (0,0) circle (3cm);
\draw[thick] (-9,0) circle (3cm);

\node at (0, 2.1) {$S_{i,j}'$};
\node at (0,0) {$V_{i,j}'$};
\node at (3.9,1.5) {$S_{i,j}''$};
\node at (-9,1.9) {$T_{i,j}$};

\draw[->]        (3.2,1.5)   -- (2.2,1.2);

\draw[black]  (-2.6,1.5) -- (2.6,1.5);
\draw[black,dashed]  (-2.7,1.0) -- (2.7,1.0);

\draw[black]  (-11.85,1.0) -- (-6.15, 1.0);

\draw[dotted, black]  (-9,3) -- (0,3);
\draw[dotted, black]  (-6.3,1.0) -- (-2.7,1.0);
\node at (0, 4) {\textbf{Partition of $V_{i,j}$}};
\node at  (-9,4) {\textbf{Partition of $V_{0,j}$}};
\end{tikzpicture}
\caption{The partition of a pair $(V_{0,j},V_{i,j})$. The pair $(S_{i,j}' \sqcup S_{i,j}'',T_{i,j} )$ is shown to be $(\epsilon_2, d/7)$-super-regular.}
\label{fig:partition}
\end{figure}

We claim that $\bigl(T_{i,j}, S_{i,j}\bigr)$ is $\bigl(\epsilon_2,d/7\bigr)$-super-regular. 
In order to do so, we must show that, for $T=T_{i,j}$ and $S=S_{i,j}$ (with $S'=S_{i,j}'$ and $S''=S_{i,j}''$), the following two conditions hold:
\begin{itemize}
    \item for all sets $X \subseteq S$, $Y\subseteq T$ such that $|X| \geq \epsilon_2 |S|$ and $|Y| \geq \epsilon_2 |T|$, we have $d_{G}(X,Y) >d/7$;
    \item for all $a \in S$ and $b \in T$, we have $\deg_{G}(a) > (d/7)|T|$ and $\deg_{G}(b) >(d/7)|S|$.
\end{itemize}
Let $X' = X \cap S'$ and $X''=  X \cap S''$ so that $X = X' \sqcup X''$ such that
\begin{align*}
    |X|     &\geq   \epsilon_2 |S| = \bigl(\epsilon_2/t\bigr) L_1 \\
    |Y|     &\geq   \epsilon_2 |T| = \bigl(\epsilon_2/t\bigr) L_1 .
\end{align*}
Then, since $\bigl|X''\bigr|\leq \bigl|S''\bigr|<\eta L_1$,
\begin{align*}
    \bigl|X'\bigr| = \bigl|X\bigr| - \bigl|X''\bigr| \geq \bigl|X\bigr| - \bigl|S''\bigr| \geq \bigl|X\bigr| - \eta \, L_1 \geq 12 \epsilon_1 L_1,
\end{align*}
because $0 < \eta \ll \epsilon_1, \ll \epsilon_2 \ll 1/t$. 
Also, $|Y| \geq \epsilon_2 |T| \geq 12 \epsilon_1 L_1$. Since $\bigl(V_{0,j}, V_{i,j}\bigr)$ is $\bigl(12\epsilon_1, d/6\bigr)$-super-regular, we obtain: 
\begin{align*}
    e\bigl(X,Y\bigr) \geq e\bigl(X',Y\bigr) &> \frac{d}{6}\bigl|X'\bigr|\bigl|Y\bigr| \\
    d(X,Y) &> \frac{d}{6} \frac{|X'|}{|X|} \geq \frac{d}{6}\biggl( 1 - \frac{|X''|}{|X|}\biggr) \geq \frac{d}{6}\Bigl(1 - \frac{\eta L_1}{\bigl(\epsilon_2/t\bigr) L_1} \Bigr)> \frac{d}{7}.
\end{align*}

This verifies the first bullet point of super-regularity. 
As to the second bullet point, we first consider $a\in S$.

For all $a\in S$, $\mathbb{E}\bigl(\deg(a,T)\bigr) \geq (d/6)\bigl(L_1/t\bigr)$. 
So using a tail bound of the hypergeometric distribution (see~\cite{JLR}, Theorem 2.10), we obtain 
 \begin{align*}
     \prob\biggl(\deg(a,T)< \frac{d}{7}
        \frac{L_1}{t} 
 \biggr) < 2\exp\biggl(-c\,\frac{d}{6}\frac{L_1}{t}\biggr).
 \end{align*}
for some positive real constant $c$. By the union bound, the probability that there exists any such vertex is at most 
\begin{align*}
    2 \exp\biggl(\log|S| -c\,\frac{d}{6}\frac{L_1}{t} \biggr) =  \exp\biggl(\log\Bigl(\frac{L_1}{t}\Bigr) -c\,\frac{d}{6}\frac{L_1}{t} \biggr) =o(1).
\end{align*}

For all $b\in T$, $\mathbb{E}\bigl(\deg(b,S)\bigr) \geq \mathbb{E}\bigl(\deg(b,S')\bigr) \geq(d/6)\bigl(L_1/s\bigr)$. Since $\deg(b,S')$ follows a hypergeometric distribution, and $1/s <1/t$ we obtain, that the probability that there exists a vertex for which $\deg(b,S) < (d/7) L_1/t$ is at most
\begin{align*}
    2 \exp\biggl(\log|T| -c\,\frac{d}{6}\frac{L_1}{s} \biggr)=o(1).
\end{align*}



We have that $(S,T)$ is $(\epsilon_2,d/7)$-super-regular and $|S| = |T|$, so we can apply the Blow-Up Lemma (Lemma~\ref{lemma: blow-up}) to complete the tiling.

\section{High-degree bipartite graphs}\label{section: alpha large}
In this section we will discuss the regime in which $\alpha=\bigl(\frac{1}{n}\bigl\lceil\frac{1}{2}n\bigr\rceil-\frac{1}{n}\bigr)$. 
We will provide an example to show that the probability threshold for a perfect $K_{h,h}$-tiling in $G_{n;\alpha}\cup G_{n,n,p}$ is at least $n^{-(h+1)/h}$ for any $h\geq 1$. 
Recall that Theorem~\ref{thm:threshold} requires the higher probability threshold of $p(n)=n^{-(2h-1)/h^2}$ for any constant $\alpha\in \bigl(0,1/(2h)\bigr)$. 



\subsection{An extremal example}
We give a balanced bipartite graph $G$ with $\delta(G) = \bigl(\frac{1}{n}\bigl\lceil\frac{1}{2}n\bigr\rceil-\frac{1}{n}\bigr)n$ such that if  $p(n) = o\bigl(n^{-(h+1)/h}\bigr)$, then \whp~$G \cup G_{n,n,p}$ does not have a perfect $K_{h,h}$-tiling. 

Let $G$ have vertex classes $A= A_{1} \sqcup A_2$ and $B= B_{1} \sqcup B_{2}$ with $\bigl|A_1\bigr| = \bigl|B_1\bigr|=\bigl\lfloor\frac{1}{2}n\rfloor+1$ and $\bigl|A_2\bigr| = \bigl|B_2\bigr|= \bigl\lceil\frac{1}{2}n\bigr\rceil-1$. 
The graph $G$ is defined to have all edges in the pairs $\bigl(A_1, B_2\bigr)$ and $\bigl(A_2, B_1\bigr)$ and no other edges. 
Observe that $\delta(G) = \bigl(\frac{1}{n}\bigl\lceil\frac{1}{2}n\bigr\rceil-\frac{1}{n}\bigr)n$. 

Since $|A_1|+|B_1|>|A_2|+|B_2|$, any perfect $K_{h,h}$-tiling requires there to be at least one copy of $K_{h,h}$ that uses at least $h+1$ vertices from $A_1\cup B_1$.
However, in order for this to occur, it would require this copy to either have a $K_{2,2}$ or a $K_{1,h}$ in $G_{n,n,p}\bigl[A_1\cup B_1\bigr]$.
We will show that if $p(n) = o\bigl(n^{-(h+1)/h}\bigr)$, then this fails to occur, \whp.

Let $X$ be the number of copies of $K_{2,2}$ in $G_{n,n,p}$ and $Y$ be the number of copies of $K_{1,h}$ in $G_{n,n,p}$, then 
$$\E[X] = \Theta \bigl(n^4p^4 \bigr) \qquad\mbox{ and }\qquad \E[Y] = \Theta \bigl(n^{h+1}  p^{h}\bigr) .$$
By Markov's inequality (Lemma~\ref{lemma: markov}),
$$\prob(X \geq 1) \leq \E[X] =o(1) \qquad\mbox{ and }\qquad \prob(Y \geq 1) \leq \E[Y] =o(1). $$
Therefore, \whp~no copy of $K_{2,2}$ or $K_{1,h}$ exists in $G_{n,n,p}[A_1 \sqcup B_1]$. 
Consequently, \whp~$G\cup G_{n,n,p}$ has no perfect $K_{h,h}$-tiling.

\section{Concluding remarks}\label{sec:conclusion}
In this paper we have determined the probability threshold for the existence of a perfect $K_{h,h}$-tiling in the randomly perturbed bipartite graph when $0
<\alpha < 1/(2h)$ for all $h\geq1$. Following a trend in randomly perturbed graph theory~\cite{HMT, bottcher2023triangles, AKRcliques}
we have also determined a lower bound for the threshold in the high degree regime $\alpha = \bigl(\frac{1}{n}\bigl\lceil\frac{1}{2}n\bigr\rceil-\frac{1}{n}\bigr)$. The main obstacle one encounters in attacking the upper bound using the same techniques as in the proof of Theorem~\ref{thm: main} is the apparent need a mechanism by which one can find a linear (in $n$) sized partial $K_{h,h}$-tiling. We leave this as an open question. 

Moreover, given that there is a multipartite analog to the Hajnal-Szemer\'{e}di theorem \cite{keevash2015multipartite}, it is also natural to consider a multipartite analog to Balogh-Treglown-Wagner~\cite{balogh2019tilings}. A natural question is to determine the threshold for the existence of a $K_{h}$-tiling in an perturbed random $h$-partite graph, for $h\geq 3$. We conjecture that this probability threshold coincides with Theorem 1.3 of~\cite{balogh2019tilings}. 

\bibliographystyle{abbrv}
\bibliography{bibfile}

\section{Appendix}\label{sec:appendix}
We will prove that if $\alpha>0$ and if $p\geq C n^{-1}$, then $G_{n;\alpha}\cup G_{n,n,p}$ contains a perfect matching whp. This is the special case of Theorem~\ref{thm: main} for $h=1$.

We verify that the strong Hall condition is met by demonstrating that whp there does not exists a set of vertices $S\subset A$ for which $|S| = |N(S)|+1$. Note that by our assumption on $G_{n;\alpha}$ we must only consider sets of size at least $\alpha n$. For a subset $S \subset A$, let $X_S$ be the indicator variable for the event that $|S|>|N(S)|$. Let $X := \sum_{S \subset A} X_S$. 
First note that $|S|\leq n-\alpha n$. This is because the minimum degree condition guarantees that if $S'\subseteq A$ has size greater than $n-\alpha n$ then $N(S') = B$ by the pigeonhole principle. 

By Markov's inequality,
\begin{align*}
        \Pr(X \geq 1) \leq \E[X] 
        &= \sum \E[X_i]\\
        &\leq \sum_{\ell=\alpha n+1}^{n-\alpha n} \binom{n}{\ell}\sum_{j=\alpha n}^{\ell-1}\binom{n}{j}(1-p)^{\ell(n-j)} \\
        &\leq (2^n)^2 \sum_{\ell=\alpha n +1}^{n-\alpha n}(1-p)^{\ell n} \sum_{j=\alpha n}^{\ell - 1} (1-p)^{-\ell j}\\
        &\leq 2^{2n} \sum_{\ell=\alpha n +1}^{n-\alpha n}(1-p)^{\ell n} \cdot n (1-p)^{-\ell^2}\\
        &\leq n^2 2^{2n} (1-p)^{n^2/4}\\
        &\leq \exp\bigl\{2\ln n + 2n\ln 2 -p n^2/4\bigr\}\\
    \end{align*}
    So in fact, $C> 8 \ln 2$ suffices. 
\end{document}